\documentclass[11pt]{amsart}

\usepackage{cmap}
\usepackage[T2A]{fontenc}
\usepackage[utf8]{inputenc}
\usepackage[english]{babel}
\usepackage{amsfonts,amssymb}
\usepackage{hyperref}
\usepackage{stmaryrd}
\usepackage[left=1.0cm,right=1.0cm,top=1.0cm,bottom=1.0cm]{geometry}

\usepackage{yhmath}
\usepackage{mathrsfs}

\usepackage[new]{old-arrows} 
\usepackage{mathtools} 

\usepackage{tikz-cd}

\newcommand{\bigwedgem}[1]{\mathord{\raisebox{2pt}
{\hbox{$\scriptstyle{\bigwedge^{\!#1}}$}}}}
\newcommand\blank{\mathord{\hbox to 1.5ex{\hrulefill}}\,}

\newtheorem{theorem}{Theorem}
\newtheorem{lemma}[theorem]{Lemma}
\newtheorem{proposition}[theorem]{Proposition}
\newtheorem{corollary}[theorem]{Corollary}
\theoremstyle{definition}
\newtheorem{remark}[theorem]{Remark}
\theoremstyle{definition}

\theoremstyle{definition}
\newtheorem{definition}[theorem]{Definition}
\theoremstyle{definition}
\newtheorem{claim}[theorem]{Claim}
\newtheorem{observation}[theorem]{Observation}

\def\<{\vartriangleleft}
\def\>{\vartriangleright}

\def\XX{{\bf X}}
\def\YY{{\bf Y}}
\def\ZZZ{{\bf Z}}
\def\XXX{\textbf{X}_{\infty}}

\def\HHPf{\mathbf{HPf}}
\def\PPf{\mathbf{Pf}}
\def\HPf{\mathrm{HPf}}
\def\Pf{\mathrm{Pf}}
\def\hpf{\mathrm{hpf}}
\def\pf{\mathrm{pf}}

\def\GGL{\mathrm{GL}_{\infty}}
\def\SI{\bigwedgem{\centerdot}}
\def\SIVar{\bigwedgem{\centerdot}\mathrm{Var}_{\mathbb{K}}}
\def\Wedge{\bigwedgem{\frac{\infty}{2}}V_{\infty}}
\def\WedgeD{\bigwedgem{\frac{\infty}{2}}V_{\infty}^{*}}
\def\WedgeDU{(\bigwedgem{\frac{\infty}{2}}V_{\infty})^{*}}
\def\Wedgegood{(\bigwedgem{\frac{\infty}{2}}V_{\infty})^{*}_{\text{good}}}

\def\GL{\mathrm{GL}}

\def\Gr{\mathrm{Gr}}
\def\GGr{{\bf Gr}}
\def\SGr{\mathcal{S}\mathrm{Gr}}

\def\NN{\mathbb{N}}

\def\ZZ{\mathbb{Z}}
\def\ZZO{\mathbb{Z}_{\geqslant 0}}
\def\ZZX{\mathbb{Z}^{\times}}

\def\sgn{\mathrm{sgn}}

\def\Sym{\mathrm{Sym}}
\def\KK{\mathbb{K}}
\def\Plu{\mathrm{PluVar}_{\mathbb{K}}}
\def\Var{\mathrm{Var}_{\mathbb{K}}}
\def\Vec{\mathrm{Vect}_{\mathbb{K}}}

\def\rk{\mathrm{rk}}
\def\id{\mathrm{id}}

\def\AA{\mathcal{A}}

\def\textblue<#1>{\textbf{\textcolor{blue}{#1}}}

\AtEndDocument{%
  \par
  \medskip
  \begin{tabular}{@{}l@{}}%
    \textsc{Department of Mathematics, University of Michigan, Ann Arbor, MI}\\
    \textit{E-mail address}: \texttt{inekras@umich.edu}
  \end{tabular}}

\title{Dual Infinite Wedge is $\mathrm{GL}_{\infty}$-equivariantly noetherian}
\author{Ilia Nekrasov}


\begin{document}

\sloppy 
\maketitle

\begin{abstract}
We prove the (equivariant) noetherian property for a wide class of varieties generalizing the class of Pl\"ucker varieties (Theorem \ref{Thm:MainThm}). It improves previous results of Draisma--Eggermont who treated the case of bounded Pl\"ucker varieties. Key ingredient of our proof is the constructive proof of the equivariant noetherianity for the hyper-Pfaffians (Theorem \ref{thm:HPfNoeth}) which implies the equivariant noetherianity of the dual infinite wedge. 
\end{abstract}


\section{Introduction}

A Pl\"ucker variety is a rule $\XX$ that assigns to a pair of a nonnegative number $p$ and a vector space $V$ an algebraic variety $\XX_{p}(V)$. These varieties are required to satisfy some compatibility properties (see Definition \ref{def:PluVar}); for instance, for every $p$ the assignment $\XX_{p}(\_)$ is a functor on a category of vector spaces. 

Pl\"ucker varieties were introduced by Draisma and Eggermont \cite{DraEgger}. These varieties give a powerful tool for proving existence of uniform bounds for degrees of equations for varieties $\XX_{p}(V)$. The most important implications of the work of Draisma--Eggermont include uniform bounds for degrees of the equations for secant and tangential varieties of Grassmannians. 

More generally, their noetherianity result \cite[Theorem 1]{DraEgger} implies that the degrees of equations for $\XX_{p}(V)$ are bounded (for all $p$ and $V$) for the so-called bounded Pl\"ucker varieties $\XX$ (Definition \ref{def:bddPluVar}). The boundedness condition for Pl\"ucker varieties originates in the work of Snowden on noetherianity for $\Delta$-modules \cite{SnowDelta}: $\Delta$-modules are ideals of equations for symmetric power counterparts of Pl\"ucker varieties.

The paper \cite{DraEgger} raises the natural question whether the same noetherianity result holds for unbounded Pl\"ucker varieties. Theorem \ref{Thm:MainThm}, our first main result, answers affirmatively to this question. 

Instead of working with general Pl\"ucker varieties, we introduce a notion of a $\SI$-variety (Definition \ref{def:SIVar}), which generalizes and at the same time simplifies the notion of a Pl\"ucker variety. Our main result (Theorem \ref{Thm:MainThm}) states the topological noetherianity result for the class of $\SI$-varieties.

The algebraic noetherianity result (a direct analog of the result for $\Delta$-modules) for $\SI$-varieties does not hold. It fails even for the Grassmannian $\SI$-variety $\GGr$: there are infinitely many (combinatorially different) types of Pl\"ucker relations. The paper of R. Laudone \cite{Laud18} follows this direction.

Our work fits in the broader context of noetherianity results for large algebraic structures, e.g.,  \cite{DraKutt, CEF, SamSnowGro, SamSymmIdeals}. Namely, the topological noetherianity for algebraic representations of infinite classical groups is proven (originally for $\GGL$ in \cite{DraismaPoly}, and for other groups in \cite{EggerSnowInfGrp}). The dual (unrestricted) infinite wedge is a nontrivial inverse limit of algebraic $\GGL$-representations. It is one of the most interesting and used in mathematics literature space among non-algebraic representations of $\GGL$. Theorem \ref{thm:WedgeNoeth}, our second main result, states the $\GGL$-equivariant noetherianity of the infinite wedge.



After completing this work, we learned that A.~Bik, J.~Draisma, and R.~Eggermont had obtained similar results in unpublished work. In a forthcoming paper with A.~Bik, J.~Draisma, and R.~Eggermont we consider applications of the main results of this paper to sequences of varieties with contraction morphisms only.

\subsection{Main results} The following theorem is our first main result.
\begin{theorem}\label{Thm:MainThm}
For any $\SI$-variety ${\bf X}$ there exists a $p_0 \in \mathbb{Z}_{\geqslant 0}$ and a finite dimensional vector space $V_{0}$ such that $\XX$ is defined set-theoretically by pullbacks of equations for $\XX_{p_0}(V_0)$. In particular, the degrees of equations defining the varieties $\XX_{p}(V)$ are bounded. 
\end{theorem}

The crucial statement for the proof is our second main result:

\begin{theorem} \label{thm:WedgeNoeth}
The dual unrestricted infinite wedge $\WedgeDU$ is $\GGL$-equivariant noetherian. Unfolding, any descending chain of closed $\GGL$-invariant subsets of $\WedgeDU$ stabilizes.
\end{theorem}

\subsection{Plan of the proof}

Following the logic of \cite{DraEgger}, we prove Theorem \ref{Thm:MainThm} via the following steps:

\begin{enumerate}
    \item We define finite dimensional hyper-Pfaffian varieties $\HPf^{(m,l)}$ and their duals $\HPf^{(r,s), \star}$. Collecting these varieties, we form the Pl\"ucker varieties $\HHPf^{(m,l)}$, $\HHPf^{(r,s), \star}$, and the limiting forms for them $\HHPf^{(m,l)}_{\infty}$, $\HHPf^{(r,s), \star}_{\infty}$. Taking into account both constructions, we build up a two-sided hyper-Pfaffian Pl\"ucker variety $\HHPf^{(m,k), (r,s)}$ and its limiting form $\HHPf^{(m,k), (r,s)}_{\infty}$.
    
    \item We prove (Theorem \ref{thm:InclToHPf}) that for any proper Plücker variety $\XX$ there exist pairs $(m,l),(r,s)$ such that $$\XXX \subseteq \HHPf^{(m,l), (r,s)}_{\infty}.$$
    
    \item We prove (Theorem \ref{thm:HPfNoeth}) that for any pairs $(m,l)$, $(r,s)$ the limiting variety $\HHPf^{(m,l), (r,s)}_{\infty}$ is $\GGL$-noetherian. 
    
    \item Finally, Theorem \ref{thm:WedgeNoeth} follows from Theorem \ref{thm:InclToHPf}, Proposition \ref{prop:filtration}, and Theorem \ref{thm:HPfNoeth}; Theorem \ref{Thm:MainThm} follows from this result immediately (Corrolary \ref{cor:MainThm}).
\end{enumerate}

\subsection{Notation and conventions}


In what follows, we work over a fixed field $\KK$ of characteristic 0. 
Besides algebraic varieties, more often than not we consider affine cones over projective varieties. For instance, by the Grassmannian $\Gr(2,4)$ we mean the affine cone over the actual projective variety inside the vector space $\bigwedgem{2}\KK^{4}$. 


The vector space generated by set of vectors $\{e_{i}\}_{i \in I}$ we denote by $\langle e_{i} \rangle_{i\in I}$ or by $\KK^{I}$ when we want to emphasize the generating set $I$ only. Also, for a vector space $V$ we denote by $[V]$ the given set of basis vectors, when this makes sense, e.g., $[\KK^{n}] = [n]$ and, more generally, $[\KK^{I}] = I$.




By $\ZZX$ we denote the set $\ZZ\backslash\{0\}$. For pairs of natural numbers we use the following partial order: $(n,p) \< (N,P)$ if and only if $n<N$, $p \leqslant P$ or $n \leqslant N$, $p < P$.

\subsection*{Acknowledgements.} The author thanks Andrew Snowden for extremely useful discussions and countless reviews of the text.

\section{\texorpdfstring{$\SI$}{TEXT}-varieties and equivariant noetherianity}

\subsection{Infinite wedge}
Consider for any $n, p \in \ZZO$ a set $[n,p] = \{-n,\dots-1, 1, \dots, p\}$ and a vector space $V_{n, p} = \KK^{[n,p]} = \langle e_{-n},\dots, e_{-1}, e_{1}, \dots e_{p} \rangle$ of dimension $n+p$. Then by $V_{\infty}$ we denote the direct limit of $V_{n, p}$ with natural inclusions:
$$V_{\infty} := \varinjlim_{n, p} V_{n,p} = \langle \dots, e_{-2}, e_{-1}, e_{1}, e_{2}, \dots\rangle_{\KK}.$$
Let $\langle\,, \rangle:\;V_{\infty} \times V_{\infty} \rightarrow \KK$ be a bilinear form on $V_{\infty}$ given by $\langle e_{i}, e_{j} \rangle = \delta_{-i, j}$. The restrictions of the form $\langle\,, \rangle$ to the spaces $V_{n,p}$ identify the dual space $V_{n,p}^*$ with $V_{p,n}$.


Next, we consider the exterior powers $\bigwedgem{p}V_{n,p}$. We denote the basis vectors for this space by $e_{I} = e_{i_{1}}\wedge\dots \wedge  e_{i_{p}}$ where $I = \{ i_1 < \dots < i_{p}\} \subset [n,p]$. The following maps between the exterior powers are of a particular interest for us:
\begin{itemize}
    \item $i_{n, p}: \bigwedgem{p}V_{n,p} \hookrightarrow \bigwedgem{p}V_{n+1,p}$ is induced by a natural inclusion $V_{n,p} \hookrightarrow V_{n+1,p}$ ;
    
    \item $j_{n,p}: \bigwedgem{p}V_{n,p} \hookrightarrow \bigwedgem{p+1}V_{n,p+1}$ is multiplication by $e_{p+1}$, i.e., $j_{n,p}(\omega) = \omega \wedge e_{p+1}$;
    
    \item $i_{n, p}^{\dagger}: \bigwedgem{p}V_{n+1,p} \twoheadrightarrow \bigwedgem{p}V_{n,p}$ is the dual map to $j_{p,n}$, i.e., $i_{n,p}^{\dagger} = (j_{p, n}^*)$;
    
    \item $j_{n,p}^{\dagger}: \bigwedgem{p+1}V_{n,p+1} \twoheadrightarrow \bigwedgem{p}V_{n,p}$ is the dual map to $i_{p,n}$.
\end{itemize}
Explicit formulas for the maps $i_{n, p}^{\dagger}$ and $j_{n,p}^{\dagger}$ are presented in the proof of Lemma \ref{SetTheoretical42}.


We note that $j_{n+1, p} \circ i_{n,p} = i_{n, p+1} \circ j_{n,p}$ and $j_{n, p}^{\dagger} \circ i_{n,p+1}^{\dagger} = i_{n, p}^{\dagger} \circ j_{n+1,p}^{\dagger}$ or, reformulating, the two diagrams -- the first one with $i_{n, p}$ and $j_{n,p}$ maps, and the second one with the corresponding $\dagger$-maps -- are commutative, see Fig.\ref{dia:Main}. 

\begin{figure}
    \centering
\begin{tikzcd}
  \bigwedgem{0}V_{0,0} \arrow[dd, hook, xshift=-2.3ex, "i_{0,0}" description] \arrow[dd, twoheadleftarrow, xshift=1.3ex, "i_{0,0}^{\dagger}" description] \arrow[rr, hook, yshift=1.3ex, "j_{0,0}" description] \arrow[rr, twoheadleftarrow, yshift=-1.3ex, "j_{0,0}^{\dagger}" description] & & \bigwedgem{1}V_{0,1} \arrow[dd, hook, xshift=-2.3ex, "i_{0,1}" description] \arrow[dd, twoheadleftarrow, xshift=1.3ex, "i_{0,1}^{\dagger}" description] \arrow[rr, hook, yshift=1.3ex, "j_{0,1}" description] \arrow[rr, twoheadleftarrow, yshift=-1.3ex, "j_{0,1}^{\dagger}" description]  & & \bigwedgem{2}V_{0,2} \arrow[dd, hook, xshift=-2.3ex, "i_{0,2}" description] \arrow[dd, twoheadleftarrow, xshift=1.3ex, "i_{0,2}^{\dagger}" description] \arrow[rr, hook, yshift=1.3ex, "j_{0,2}" description] \arrow[rr, twoheadleftarrow, yshift=-1.3ex, "j_{0,2}^{\dagger}" description]& & \dots \\
                                                         & &                                                         & &                      & &       \\
  \bigwedgem{0}V_{1,0} \arrow[dd, hook, xshift=-2.3ex, "i_{1,0}" description] \arrow[dd, twoheadleftarrow, xshift=1.3ex, "i_{1,0}^{\dagger}" description] \arrow[rr, hook, yshift=1.3ex, "j_{1,0}" description] \arrow[rr, twoheadleftarrow, yshift=-1.3ex, "j_{1,0}^{\dagger}" description] & & \bigwedgem{1}V_{1,1} \arrow[dd, hook, xshift=-2.3ex, "i_{1,1}" description] \arrow[dd, twoheadleftarrow, xshift=1.3ex, "i_{1,1}^{\dagger}" description] \arrow[rr, hook, yshift=1.3ex, "j_{1,1}" description] \arrow[rr, twoheadleftarrow, yshift=-1.3ex, "j_{1,1}^{\dagger}" description]  & & \bigwedgem{2}V_{1,2} \arrow[dd, hook, xshift=-2.3ex, "i_{1,2}" description] \arrow[dd, twoheadleftarrow, xshift=1.3ex, "i_{1,2}^{\dagger}" description] \arrow[rr, hook, yshift=1.3ex, "j_{1,2}" description] \arrow[rr, twoheadleftarrow, yshift=-1.3ex, "j_{1,2}^{\dagger}" description]& & \dots \\
                                                         & &                                                         & &                      & &       \\
  \bigwedgem{0}V_{2,0} \arrow[dd, hook, xshift=-2.3ex, "i_{2,0}" description] \arrow[dd, twoheadleftarrow, xshift=1.3ex, "i_{2,0}^{\dagger}" description] \arrow[rr, hook, yshift=1.3ex, "j_{2,0}" description] \arrow[rr, twoheadleftarrow, yshift=-1.3ex, "j_{2,0}^{\dagger}" description] & & \bigwedgem{1}V_{2,1} \arrow[dd, hook, xshift=-2.3ex, "i_{2,1}" description] \arrow[dd, twoheadleftarrow, xshift=1.3ex, "i_{2,1}^{\dagger}" description] \arrow[rr, hook, yshift=1.3ex, "j_{2,1}" description] \arrow[rr, twoheadleftarrow, yshift=-1.3ex, "j_{2,1}^{\dagger}" description]  & & \bigwedgem{2}V_{2,2} \arrow[dd, hook, xshift=-2.3ex, "i_{2,2}" description] \arrow[dd, twoheadleftarrow, xshift=1.3ex, "i_{2,2}^{\dagger}" description] \arrow[rr, hook, yshift=1.3ex, "j_{2,2}" description] \arrow[rr, twoheadleftarrow, yshift=-1.3ex, "j_{2,2}^{\dagger}" description]& & \dots \\
                                                         & &                                                         & &                      & &       \\
  \vdots  & & \vdots & & \vdots \\
\end{tikzcd}
    \caption{Two commutative diagrams: the one with $i$- and $j$-maps, and the one with $i^{\dagger}$- and $j^{\dagger}$-maps}
    \label{dia:Main}
\end{figure}

\begin{definition}
The \textit{infinite wedge} $\Wedge$ is the direct limit of the spaces $\bigwedgem{p}V_{n,p}$ with respect to the transition maps $i_{n, p}$ and $j_{n,p}$:
$$\Wedge := \varinjlim_{n,p} \bigwedgem{p}V_{n,p}.$$
Analogously, we define the (restricted) \textit{dual infinite wedge} 
$$\WedgeD := \varinjlim\limits_{n,p} \bigwedgem{p}V_{n,p}^* = \varinjlim\limits_{n,p} \bigwedgem{n}V_{p, n}.$$
It is isomorphic to the infinite wedge $\Wedge$.

The \textit{unrestricted dual infinite wedge} $\WedgeDU$ is an uncountable dimensional vector space that is defined with the inverse limit of the spaces $\bigwedgem{p}V_{n,p}$ with respect to the transition maps $i_{n,p}^{\dagger}$ and $j_{n,p}^{\dagger}$:
$$\WedgeDU := \varprojlim\limits_{n, p} \bigwedgem{p}V_{n,p}.$$
\end{definition}
\begin{remark}
The space $\Wedge$ was introduced in mathematical physics by Jimbo et al. \cite{SatoSato82}. We call it the infinite wedge following \cite{DraEgger}, but this space has many other names: semi-infinite wedge \cite{Stern95}, half-infinite wedge \cite{Rios-Zertuche13}, charge-0 fermionic Fock space \cite{MuthiahWeekesYacobi18} and others.
\end{remark}

By $\GL_{n,p}$ and $\GGL$ we denote the groups $\GL(V_{n,p})$ and $\bigcup_{n, p \in \ZZO} \mathrm{GL}(V_{n, p})$ respectively. This group acts on the spaces $V_{\infty}$, $\Wedge$, $\WedgeD$, and $\WedgeDU$ combining actions of $\GL_{n,p}$ on $V_{n,p}$ for all $n,p$.

From the definition we can see that basis elements of the space $\Wedge$ have the form 
$$e_{i_{1}, i_{2}, \dots}:= e_{i_{1}} \wedge e_{i_{2}} \wedge \dots \text{ with }i_{k}=k\text{ for }k\gg 0.$$
The space $\Wedge$ is a proper subspace of $\WedgeDU$. Indeed, for any element $w = (w_{n,p})$ of the former space, because of the identity $i^{\dagger}_{n,p}\circ i_{n,p} = j^{\dagger}_{n,p} \circ j_{n,p} = \id_{V_{n,p}}$, the same set $(w_{n, p})$ is an element of the latter space. However, due to the dimension count, the spaces are not equal. For instance, the vector $v = (v_{n,p}) \in \WedgeDU$ with $v_{n,p} = (e_{1}+e_{2})\wedge \dots \wedge(e_{p-1}+e_{p})\wedge e_{p}$ for any $n\geqslant 0$ do not belong to the space $\Wedge$:
$$v = (e_{1}+e_{2}) \wedge (e_{2}+e_{3}) \wedge (e_{3}+e_{4}) \wedge \dots \in \WedgeDU\backslash\Wedge.$$

\subsection{Pl\"ucker and \texorpdfstring{$\SI$}{TEXT}-varieties}\label{sec:PluSIVar}
First we reproduce the original definition of Pl\"ucker variety by Draisma--Eggermont~\cite{DraEgger}.
\begin{definition}\label{def:PluVar}
A \textit{Pl\"ucker variety} is a sequence $(\XX_{p})_{p \in \ZZO}$ of functors from the category $\Vec$ to the category of varieties $\Var$ satisfying the following axioms:
\begin{enumerate}
    \item For all vector spaces $V$ and for all $p \in \ZZO$, the variety $\XX_{p}(V)$ is a closed subvariety of $\bigwedgem{p}V$.
    
    \item For all $p \in \ZZO$ and for all linear maps $\phi: V \rightarrow W$, the map $\XX_{p}(\phi): \XX_{p}(V) \rightarrow \XX_{p}(W)$ coincides with the restriction of $\bigwedgem{p}\phi$.
    
    \item If $V$ is a vector space of dimension $n+p$ with $n, p \in \ZZO$, and $\star: \bigwedgem{p}V \rightarrow \bigwedgem{n}V^{*}$ is the Hodge dual, then the transformation $\star$ maps $\XX_{p}(V)$ into $\XX_{n}(V^*)$.
\end{enumerate}

Pl\"ucker varieties form a category in a natural way; we denote it by $\Plu$. 
\end{definition}

Next we introduce the more general notion of $\SI$-variety.

\begin{definition}\label{def:SIVar}
A \textit{$\SI$-variety} $\XX$ is a set of closed varieties $\XX = \{\XX_{n, p} \subseteq \bigwedgem{p}V_{n,p}\}_{n, p \in \ZZO}$ satisfying the following conditions:
\begin{enumerate}
    \item[(i)] For all $p, n\in \ZZO$, the maps $i_{n,p}$ and $j_{n,p}$ induce injections of $\XX_{n,p}$ into $\XX_{n+1,p}$ and $\XX_{n,p}$ into $\XX_{n,p+1}$ respectively;
    \item[(ii)] For all $p, n\in \ZZO$, the maps $i_{n,p}^{\dagger}$ and $j_{n,p}^{\dagger}$ induce surjections of $\XX_{n+1,p}$ onto $\XX_{n,p}$ and $\XX_{n,p+1}$ onto $\XX_{n,p}$ respectively;
    \item[(iii)] For all $p, n\in \ZZO$, the variety $\XX_{n, p}$ is $\GL_{n,p}$-invariant.
\end{enumerate}
$\SI$-varieties form a category in a natural way; we denote it by $\SIVar$.
\end{definition}

\begin{remark}\label{rem:SIVarDef} In the definition, injectivity and surjectivity conditions are automatically satisfied. Indeed, $i_{n,p}$ and $j_{n,p}$ are injective, so are their restrictions to $\XX_{n,p}$. Also, $i^{\dagger}_{n,p} \circ i_{n,p} = \id_{\bigwedgem{p}V_{n,p}}$ and $j^{\dagger}_{n,p} \circ j_{n,p} = \id_{\bigwedgem{p}V_{n,p}}$, so the restrictions of $i_{n,p}^{\dagger}$ and $j_{n,p}^{\dagger}$ to the components of $\SI$-variety are surjective.
\end{remark}

\begin{proposition}\label{prop:PluToSI}
Every Pl\"ucker variety is naturally a $\SI$-variety, and $\Plu$ is a subcategory of $\SIVar$:
$$\Plu \longhookrightarrow \SIVar.$$
\end{proposition}
\begin{proof} For an arbitrary Pl\"ucker variety $\XX = (\XX_{p})$ we consider the set $\{\XX_{n, p}\}$ where $\XX_{n,p} := \XX_p(V_{n,p})$. We prove that this set is a $\SI$-variety.

The conditions for the maps $i_{n,p}$ and $i_{n,p}^{\dagger}$ follows from condition (2) for Pl\"ucker varieties. The conditions for the maps $j_{n, p}$ and $j_{n, p}^{\dagger}$ follows from conditions (2) and (3) for Pl\"ucker varieties. In detail, Lemmata 2.3 and 2.4 of \cite{DraEgger} and Remark \ref{rem:SIVarDef} above give the proof. The $\GL$-invariancy condition follows from condition (2) for Pl\"ucker varieties.
\end{proof}
\begin{remark}
The subcategory $\Plu$ is not a full subcategory in $\SIVar$. 
\end{remark}

Pfaffian and hyper-Pfaffian varieties give explicit examples of $\SI$-varieties which are not Pl\"ucker, see Section \ref{sec:SINotPlu}.

\begin{definition}
The \textit{limiting variety} $\XX_{\infty}$ for a $\SI$-variety $\XX$ is the inverse limit of the varieties $\XX_{n,p}$ with respect to the maps $i_{n,p}^{\dagger}$ and $j_{n,p}^{\dagger}$:
$$\XX_{\infty} := \varprojlim_{n,p} \XX_{n,p}.$$
$\XX_{\infty}$ is an inverse limit of affine schemes, so it is a subscheme of the affine space $\WedgeDU$. The $\GGL$-action on $\XXX$ is inherited from the $\XX_{n,p}$'s; it coincides with the restriction of the $\GGL$-action on $\WedgeDU$. The procedure of taking a limiting variety is a functor $(\_)_{\infty}: \SIVar \rightarrow \Var$.
\end{definition}

\subsection{Examples of Pl\"ucker and \texorpdfstring{$\SI$}{TEXT}-varieties} \label{sec:ExamplesSI}
We give several examples and constructions of Pl\"ucker and $\SI$-varieties:
\begin{enumerate}
    \item First examples of Pl\"ucker varieties are the trivial ones:  $\XX_{n,p} = \emptyset$ and $\XX_{n,p} = \{0\}$ for all $n, p \in \ZZO$. An opposite example is the \textit{ambient Pl\"ucker variety} with  $\XX_{n,p} = \bigwedgem{p}V_{n,p}$, $n, p \in \ZZO$. The limiting variety of the ambient Pl\"ucker variety coincides with the unrestricted dual infinite wedge $\WedgeDU$.
    \item The most popular example of a Pl\"ucker variety is the \textit{Grassmann Pl\"ucker variety} $\GGr$, defined by the sequence of Grassmannians $\GGr_{n, p} = \mathrm{Gr}(p, n+p) \varsubsetneq \bigwedgem{p}V_{n,p}$, i.e., the sequence of Grassmannians form the Pl\"ucker variety $\GGr$ in a natural manner.  The limiting variety $\GGr_{\infty}$ is classically known as the Sato Grassmannian $\SGr$ \cite{DateJimboKashiwaraMiwa81, SatoSato82, SegalWilson85}. We attentively consider this example in Section \ref{sec:Gra}.  
    \item The operations of intersection, union, join, and taking tangential variety are defined in the categories $\Plu$ and $\SIVar$. So, for example, if two Pl\"ucker varieties $\XX, \YY$ are given, we can consider the union $\XX \cup \YY$, the intersection $\XX \cap \YY$, the join $\XX+\YY$, and the tangential Pl\"ucker variety $\tau \XX$. 
\end{enumerate}

Pfaffian and, more generally, hyper-Pfaffian $\SI$-varieties, except few exceptions (see Section \ref{sec:SINotPlu}), are examples of $\SI$-varieties which are not Pl\"ucker; for definitions see Sections \ref{sec:Pfa} and \ref{sec:hyPfa} respectively.

\subsection{Maximal \texorpdfstring{$\SI$}{TEXT}-varieties} \label{sec:MaxMinSIVar}
In what follows we will need a notion of the maximal $\SI$-variety with some fixed component. This section is devoted to the construction of such varieties. In particular, the existence of such varieties follows.  

Let $X$ be a $\GL_{n,p}$-invariant subvariety inside $\bigwedgem{p}V_{n,p}$. We can ask for a description of all $\SI$-varieties $\YY$ such that $\YY_{n,p} = X$.

What are the possibilities for $\YY_{n+1, p}$? On the one hand, the $(n+1, p)$-component should contain the variety $\overline{\GL_{n+1,p}\cdot i_{n,p}(X)}$; on the other hand, the component should be contained in the variety $(i_{n,p}^{\dagger})^{-1}\left( X \right)$. The same logic is applicable for the component $\YY_{n, p+1}$ and $j$-maps.

Generalizing, we get the following statement.

\begin{proposition}\label{prop:components}
For any natural $n,p$ and any $\SI$-variety $\YY$, we have the inclusions
\begin{itemize}
    \item[a)] $$\overline{\bigcup_{i} i(\YY_{n-1, p})} \subseteq \YY_{n, p} \subseteq \bigcap_{i^{\dagger}} (i^{\dagger})^{-1}(\YY_{n-1,p}),$$
where the right intersection runs over the orbit of the map $i_{n-1,p}$ under the $\GL_{n-1,p}\times\GL_{n, p}$-action and the left union runs over the orbit of $i_{n-1,p}^{\dagger}$ under the same group;
    \item[b)] $$\overline{\bigcup_{j} j(\YY_{n, p-1})} \subseteq \YY_{n, p} \subseteq \bigcap_{j^{\dagger}} (j^{\dagger})^{-1}(\YY_{n,p-1}),$$
where the right intersection runs over the orbit of the map $j_{n,p-1}$ under the $\GL_{n, p-1}\times\GL_{n, p}$-action and the left union runs over the orbit of $j_{n,p-1}^{\dagger}$ under the same group.
\end{itemize}
\end{proposition}

Explicit examples of situations when the left and right sides do not coincide are given in Section \ref{sec:ExplHPfa}.  

\begin{definition}
We call a $\SI$-variety $\YY$ \textit{maximal with respect to the $(n,p)$-coordinate}, or just \textit{$(n,p)$-maximal}, if for all pairs $(N,P)\>(n,p)$ the right inclusions of Proposition \ref{prop:components} are equalities, i.e., 
$$\YY_{N, P} = \bigcap_{i^{\dagger}} (i^{\dagger})^{-1}(\YY_{N-1,P}) \text{ and }\YY_{N, P} = \bigcap_{j^{\dagger}} (j^{\dagger})^{-1}(\YY_{N,P-1}).$$
The equalities are consistent due to the commutativity for $i^{\dagger}$- and $j^{\dagger}$-maps.
\end{definition}
Note that if the $\SI$-variety $\XX$ is $(n,p)$-maximal then all components $\XX_{N,P}$ with $N>n$ or $P>p$ are determined by the component $\XX_{n,p}$: informally, all these components are given by (combinatorially) the same equations as the variety $\XX_{n,p} \subseteq \bigwedgem{p}V_{n,p}$. 

The following examples of maximal $\SI$-varieties clarify this point of view:
\begin{enumerate}
    \item[$(1)$] $(n,p)$-maximal variety $\XX$ with $\XX_{n, p} = \bigwedgem{p}V_{n,p}$ satisfies $\XX_{N,P} = \bigwedgem{P}V_{N,P}$ for $(N,P) \> (n,p)$;
    \item[$(2)$] $(2,2)$-maximal variety with $\XX_{2,2} = \Gr(2,4)$ coincides with the Grassmannian $\SI$-variety $\GGr$;
    \item[$(3)$] generalizing both previous examples, if the $\SI$-variety $\XX$ is $(n,p)$-maximal and the variety $\XX_{n,p}$ is defined by the ideal $I_{n,p}$, then the ideals $I_{N,P}$ for the varieties $\XX_{N,P}$ have the following form
    $$I_{N,P} = \bigcap_{k^{\dagger}} (k^{\dagger})^{*}(I_{n,p}),$$
    where $k^{\dagger}$ runs over all compositions of $i^{\dagger}$- and $j^{\dagger}$-maps of the form $\XX_{N,P} \rightarrow \XX_{n,p}$.
\end{enumerate}
Further examples are given by Pfaffians (Section \ref{sec:Pfa}) and hyper-Pfaffians (Section \ref{sec:hyPfa}).

Analogously $(n,p)$-maximal varieties, we can define the \textit{$(n, p)$-minimal} $\SI$-variety. However, we emphasize that the existence of such varieties needs a proof (one such the author knows from a private communication with Jan Draisma). Using the existence of minimal varieties, it can be proved that in general the $j^{\dagger}$-image of the $\bigcap_{j^{\dagger}} (j^{\dagger})^{-1}(\YY_{n,p-1})$ coincides with $\YY_{n,p}$. We do not need these facts, so proofs are ommitted.

\subsection{Grassmannian is equivariantly noetherian}\label{sec:Gra}

We are interested in $\SI$-varieties that are ``equivariantly noetherian''. Before formulating this property rigorously, we give a motivating example of equivariantly noetherian Pl\"ucker variety.

It is known \cite{KasmanEtAl} that an arbitrary Grassmannian $\Gr(k,n)$ can be described set-theoretically as an intersection of all possible pullbacks (of linear maps) of the smallest (nontrivial) Grassmannian $\Gr(2,4)$:
$$\Gr(k, n) = \bigcap_{\phi:\;\KK^n \rightarrow \KK^4} \phi^{-1} (\Gr(2,4)).$$
So the Pl\"ucker variety $\GGr$ is defined set-theoretically by pullbacks of equations for $\GGr_{2,2} = \Gr(2,4)$. In other words, Theorem \ref{Thm:MainThm} holds in the case $\XX = \GGr$ with $p_{0} = 2, V_{0} = \KK^{4}$. 

This description can be rephrased geometrically: a (projective) Grassmann variety coincides with a set of decomposable vectors. Equivalently, for any Grassmann variety $\GGr_{n,p}$, the set of its $\KK$-points is a union of two $\GL_{n, p}$-orbits:
$$\GGr_{n,p} = \GL_{n,p}\cdot 0 \sqcup \GL_{n,p} \cdot e_{1}\wedge\dots\wedge e_{p}.$$
Indeed, for the equivalence we note that the case $(n,p) = (2,2)$ is the smallest case when the orbit $\GL_{n,p} \cdot e_{1}\wedge\dots\wedge e_{p}$ does not coincide with the ambient space $\bigwedgem{p}V_{n,p}\backslash\{0\}$. For a detailed proof see \cite{KasmanEtAl}.

The same property can be rephrased in the language of equations as follows. It is classically known \cite[p. 211]{GriHar94} that equations for the Grassmannian $\Gr(2, n)$ (its $\KK$-points) are given by pullbacks of the equation for the Grassmannian $\Gr(2,4)$. In detail, $\Gr(2,4)$ is given by the unique equation 
$$\pf^{(2)}_{\{1234\}} := x_{12}x_{34} - x_{13}x_{24} + x_{14}x_{23},$$
where $x_{ij}$ are coordinates on the space $\bigwedgem{2}\KK^4$.
This equation is known as the Pfaffian of degree 2. It spans a 1-dimensional subrepresentation of the $\GL_{4}$-representation $\Sym^{2} \left(\bigwedgem{2}\KK^4 \right)$ with highest weight $(1,1,1,1)$. In the case of an arbitrary $\Gr(2, n)$, the equations form a subrepresentation inside $\Sym^{2} \left(\bigwedgem{2}\KK^n \right)$ with the same highest weight $(1,1,1,1)$. This subrepresentation has as a basis the following set of polynomials:
$$\pf^{(2)}_{\{ijkl\}} = x_{ij}x_{kl} - x_{ik}x_{jl} + x_{il}x_{jk} \text{ for all }\{ijkl\} \in \bigwedgem{4}[n].$$
\begin{remark}
The latter property holds for $\KK$-points only, it is wrong on an ideal-theoretic level: the decomposition of $\Sym^{2}\left( \Wedge \right)$ into irreducible has infinitely many non-isomorphic summands, see \cite[Example I.8.9(b)]{MacDonald}. Related results for ideals of secant varieties of Grassmannians are in \cite{Laud18}.
\end{remark}

All three reformulations (set-theoretic, orbit-theoretic, and equation-theoretic) can be restated for the single variety $\GGr_{\infty}$ instead of the set $\{\GGr_{n,p}\}_{n,p \in \ZZO}$. For example, the Sato Grassmannian $\GGr_{\infty} = \SGr$ coincides with a union of the $\GL_{\infty}$-orbit of the highest weight vector $e_{1,2,3,\dots} = e_{1} \wedge e_{2} \wedge e_{3} \wedge \dots \in \Wedge$ and the zero vector $0 \in \Wedge$; see \cite{MiwaJinboJimboDate00} for details and connections to the KP hierarchy (originally appeared in \cite{DateJimboKashiwaraMiwa81}).

\subsection{Bounded Pl\"ucker varieties}\label{sec:bounded}

Describing the class of equivariantly noetherian Pl\"ucker and, more generally, $\SI$-varieties is a natural problem. Pursuing this question for Pl\"ucker varieties, Draisma--Eggermont introduced the following definition.

\begin{definition}\label{def:bddPluVar}
We call a Pl\"ucker variety $\XX$ \textit{bounded} if there exists a finite dimensional vector space $W$ such that the variety $\XX_{2}(W)$ does not coincide with $\bigwedgem{2}W$.
\end{definition}

The next theorem is a central result of \cite{DraEgger}.

\begin{theorem}\label{MainThmBdd} (Equivariant noetherianity for bounded Pl\"ucker varieties) 
Let $\XX$ be a bounded Pl\"ucker variety. Any closed $\GGL$-stable subset $Z$ of $\XX_{\infty}$ is contained in a union of a finite number of $\GGL$-orbits. Reformulating, there exists a $p_0 \in \mathbb{Z}_{\geqslant 0}$ and a finite dimensional vector space $V_{0}$ such that ${\bf X}$ is defined set-theoretically by pullbacks of equations for $\XX_{p_0}(V_0)$.
\end{theorem}

\subsection{Pfaffian \texorpdfstring{$\SI$}{TEXT}-varieties}\label{sec:Pfa}

The following generalization of the Grassmann Pl\"ucker variety is essential for the proof of Draisma--Eggermont.

\begin{definition}\label{PfDef}
The \textit{Pfaffian form} $\pf_{A}^{(l)}$ of degree $l$ on a set $A$ of cardinality $2l$ is the polynomial form given by the following formula
$$\pf^{(l)}_{A} := \sum_{I_{1}\sqcup \dots \sqcup I_{l} = A} \sgn(I_{1}, \dots, I_{l}) x_{I_1}\dots x_{I_l},$$
where the summation runs over all unordered partitions $A = I_{1} \sqcup \dots \sqcup I_{l}$ into 2-sets $I_i$, and the $\sgn(I_{1}, \dots, I_{l})$ is a sign of the permutation given in the one-line form by $(I_{1}, \dots, I_{l})$.

The \textit{Pfaffian variety} $\Pf^{(l)}_{B}$ of degree $l$ on a set $B$ is the closed subvariety of $\bigwedgem{2}\KK^{B}$ given by the Pfaffians $\pf^{(l)}_{A}$ for all $A \subseteq B$ of cardinality $2l$:
$$\Pf^{(l)}_{B} = \Pf^{(l)}(\KK^{B}):= \bigcap_{A \subseteq B, |A| = 2l} \Pf^{(l)}_{A} = \bigcap_{A \subseteq B, |A| = 2l} \{ \pf^{(l)}_{A} = 0\}.$$
Generally, $\Pf^{(l)}(V)$ is a variety given by all degree-$l$ Pfaffian forms inside the vector space $\bigwedgem{2}V$.

The \textit{Pfaffian $\SI$-variety} $\PPf^{(l)}$ of degree $l$ is the $(2l-2, 2)$-maximal $\SI$-variety with 
$$\PPf^{(l)}_{2l-2, 2} = \Pf^{(l)}(V_{2l-2,2}) \subseteq \bigwedgem{2}V_{2l-2,2}.$$
It will be proven (Theorem \ref{thm:HPfComponents}) that for every $l$ such defined $\SI$-variety exists, and it is uniquely defined by the property above. 
\end{definition}

In the case $l=2$, these definitions produce Grassmannians $\PPf^{(2)}_{n,p} = \Gr(p, n+p)$, see Section \ref{sec:Gra}. So we have the equality of $\SI$-varieties $\PPf^{(2)} = \GGr$; in particular, the $\SI$-variety $\PPf^{(2)}$ is Pl\"ucker. The general Pfaffian $\PPf^{(l)}$ is also equivariantly noetherian \cite{DraEgger}, but for $l \geqslant 3$ they are not Pl\"ucker, see Proposition \ref{prop:PfPlu}.

We note that our notation for Pfaffian varieties $\Pf^{(l)}(V)$ differs from the one of Draisma--Eggermont by the shift of the argument by one: in our convention the Pfaffian forms of degree $l$ are zero on the Pfaffian variety of degree $l$.

As mentioned above, Grassmann variety is the set of all decomposable vectors in the corresponding exterior power of a vector space. Analogously, there exists a geometrical description for all Pfaffian varieties: a vector $v \in \bigwedgem{2}V$ satisfies $v \in \Pf^{(l+1)}(V)$ if and only if $v$ has rank at most $l$. In other words, the rank filtration on the second exterior power $\bigwedgem{2}V$ coincides with the one given by Pfaffians:
\begin{align*}
    \left\{ \dots \subseteq R^l(V) := \{v \in \bigwedgem{2}V:\; \rk(v) \leqslant l\} \subseteq R^{l+1}(V) \subseteq \dots  \subseteq \bigwedgem{2}V \right\}& =\\
    = \left\{ \dots \subseteq \Pf^{(l)}(V) \subseteq \Pf^{(l+1)}(V) \subseteq \dots \subseteq \bigwedgem{2}V \right\}&.
\end{align*}
Using this filtration, we can see that any bounded Pl\"ucker variety is contained in some Pfaffian $\SI$-variety $\PPf^{(l)}$. Therefore any bounded Pl\"ucker variety is equivariantly noetherian.

The goal of this paper is to prove that the Draisma--Eggermont result (equivariant noetherianity) actually holds in its most general form. Conceptually, we describe the analogous filtration on an arbitrary exterior power $\bigwedgem{p}V$ via the so-called hyper-Pfaffian varieties. Description of the filtration in general situation is presented in Section \ref{GeneralFiltration}.   





\section{Hyper-Pfaffians}\label{sec:hyPfa}

In this section we recall a natural generalization of Pfaffians, the so-called hyper-Pfaffian varieties \cite[Definition 3.4]{Bar95}. We also present here some useful properties of these varieties and give an explicit example of $\SI$-variety which is not Pl\"ucker.

\subsection{Definitions} \begin{definition}\label{HyperDef}
The \textit{hyper-Pfaffian form} $\hpf_{A}^{(m,l)}$ of degree $l$ and width $m = 2m_1 $ on a set $A$ of cardinality $|A| = ml$ is the polynomial form in $\Sym^{l}\left(\bigwedgem{m}\KK^{A}\right)$ given by the formula
$$\hpf^{(m, l)}_{A} = \hpf^{(m, l)}_{A}({\bf x}) := \sum_{I_{1}\sqcup \dots \sqcup I_{l} = A} \sgn(I_{1}, \dots, I_{l})\; x_{I_1}\dots x_{I_l},$$
where the summation runs over all unordered partitions $A = I_{1} \sqcup \dots \sqcup I_{l}$ into $m$-sets $I_i$, the $\sgn(I_{1}, \dots, I_{l})$ is a sign of the permutation given in the one-line form by $(I_{1}, \dots, I_{l})$, and ${\bf x}$ denotes the set of variables $\{x_{I}, I \in \bigwedgem{m}A\}$. We will denote the corresponding multilinear form depending on sets of variables ${\bf x}^{(1)}, \dots , {\bf x}^{(l)}$ by $\hpf^{(m,l)}_{A}({\bf x}^{(1)}, \dots , {\bf x}^{(l)})$.

The \textit{hyper-Pfaffian variety} $\HPf^{(m,l)}_{B}$ of degree $l$ and width $m$ on a set $B$ is the closed subvariety of $\bigwedgem{m}\KK^{B}$ given by the hyper-Pfaffians $\hpf^{(l)}_{A}$ for all $A \subseteq B$ of cardinality $|A| = ml$:
$$\HPf^{(m,l)}_{B} = \HPf^{(m,l)}(\KK^{B}) := \bigcap_{A \subseteq B, |A| = ml} \HPf^{(m,l)}_{A} = \bigcap_{A \subseteq B, |A| = ml} \{ \hpf^{(m,l)}_{A} = 0\}.$$
We define $\HPf^{(m,l)}(V) \subseteq \bigwedgem{m}V$ as a variety given by all degree-$l$ hyper-Pfaffian forms inside the vector space $\bigwedgem{l}V$.

The \textit{hyper-Pfaffian $\SI$-variety} $\HHPf^{(m, l)}$ of degree $l$ and width $m$ is the $(m(l-1),m)$-maximal $\SI$-variety with
$$\HHPf^{(m,l)}_{m(l-1), m} = \HPf^{(m,l)}(V_{m(l-1),m}) \subseteq \bigwedgem{m}V_{m(l-1),m}.$$
\end{definition}
\begin{remark}\label{Root} We give a couple of remarks about the definitions.
\begin{enumerate}
    \item[(1)] In the case $m=2$, we get the Pfaffians: $$\hpf^{(2,l)}_{A} = \pf^{(l)}_{A}\text{, }\HPf^{(2, l)}(V) = \Pf^{(l)}(V)\text{, and }\HHPf^{(2, l)} = \PPf^{(l)}.$$
    \item[(2)] For odd natural number $m$ we can define hyper-Pfaffian forms $\hpf^{(m,l)}$, but, because of the sign of the monomials, such forms are identically zero (or not $\GL$-invariant if we take only half of the monomials). However, this definition gives correctly defined skew-symmetric forms $\hpf^{(m,l)} \in \bigwedgem{l}( \bigwedgem{m} V)$. For instance, for $m =1$ we get the volume form: $$\hpf^{(1, n)}_{[n]} \big( (x_{i})_{1 \leqslant i \leqslant n} \big) =  x_{1}\wedge \dots \wedge x_{n}.$$
    \item[(3)] Generalizing the two previous remarks, the hyper-Pfaffian form $\hpf^{(m,l)}$ stands for an $m$-th root of the determinant. Indeed, on a set $\bigwedgem{l}\ZZ^{ml}$ the following equality holds $$\left(\hpf^{(m,l)}_{[ml]}\right)^m = \det: \bigwedgem{l}\ZZ^{ml} \rightarrow \bigwedgem{ml}\ZZ^{ml}.$$  
    \end{enumerate}
\end{remark}

For the main structural result on a general $\SI$-variety (Theorem \ref{thm:InclToHPf}) we will need the dual notion for a hyper-Pfaffian. 

\begin{definition}
The \textit{dual hyper-Pfaffian variety} $\HPf^{(r, s), \star}$ is defined as the Hodge-dual of the $\HPf^{(r, s)}$ on a dual vector space:
$$\HPf^{(r,s), \star} (V) := \star \HPf^{(r,s)} (V^{*}) \subseteq \bigwedgem{\dim V - r}V.$$
The dual hyper-Pfaffian varieties form the \textit{dual hyper-Pfaffian $\SI$-variety} $\HHPf^{(r,s), \star}$: it is the $(r,r(s-1))$-maximal $\SI$-variety with
$$\HHPf^{(r,s), \star}_{r, r(s-1)} = \HPf^{(r,s), \star}(V_{r, r(s-1)}) = \star\left(\HPf^{(r,s)}(V_{r(s-1), r})\right)\subseteq \bigwedgem{r(s-1)}V_{r,r(s-1)}.$$
A \textit{two-sided hyper-Pfaffian $\SI$-variety} $\HHPf^{(m,l),(r,s)}$ is an intersection of the hyper-Pfaffian $\HHPf^{(m,l)}$ and the dual hyper-Pfaffian $\HHPf^{(r,s), \star}$:
$$\HHPf^{(m,l),(r,s)} := \HHPf^{(m,l)} \cap \HHPf^{(r,s), \star}.$$
\end{definition}

The following result shows fundamental role of all (not only Pfaffian, but) hyper-Pfaffian forms for the exterior algebra.

\begin{proposition}\label{KeyMult}
Hyper-Pfaffian forms give the structure constants of the exterior algebra. Explicitly, if $V$ is a finite dimensional $\KK$-vector space, then for any $v_{1}, \dots, v_{l} \in \bigwedgem{m}V$ we have the equality
$$v_{1}\wedge \dots \wedge v_{l} = \sum_{A} \hpf_{A}^{(m, l)}(v_{1}, \dots, v_{l}) e_{A},$$
where $\{e_{A} := \bigwedge\limits_{i \in A} e_{i}, A\in \bigwedgem{ml}[V]\}$ is a basis for $\bigwedgem{m} V$. 
\end{proposition}
\begin{proof}
If $v_{i} = \sum\limits_{I \in \bigwedgem{m}[ V]} a_{i, I}e_{I}$, then $v_{1}\wedge\dots\wedge v_{l}$ is equal to 
$$\sum_{I_{1}, \dots, I_{l}} a_{1, I_{1}}\dots a_{l, I_{l}} e_{I_{1}} \wedge \dots \wedge e_{I_{l}} = \sum_{A\in \bigwedgem{ml}[ V]} \left( \sum_{I_{1}\sqcup\dots\sqcup I_{l} = A} a_{1, I_{1}}\dots a_{l, I_{l}} \sgn(I_{1}, \dots, I_{l})\right) e_{A}.$$
This expression coincides with the sum $\sum\limits_{A} \hpf_{A}^{(m, l)}(v_{1}, \dots, v_{l}) e_{A}$.
\end{proof}
\begin{corollary}\label{GeoHPf}
The hyper-Pfaffian variety $\HPf^{(m,l)}(V)$ coincides with the set of nilpotency degree-$l$ vectors in $\bigwedgem{m}V$:
$$\HPf^{(m,l)}(V) = \{v \in \bigwedgem{m}V:\;v^{\wedge l} = 0\}.$$
Analogously, $\HPf^{(r,s), \star}(V) = \{v \in \bigwedgem{\dim V - r}V:\;(\star v)^{\wedge s} = 0\}$. 
\end{corollary}
\begin{proof}
Proposition \ref{KeyMult} for $v_{1} = \dots = v_{l} = v$ proves the statement.
\end{proof}

\begin{remark}
Morally, Proposition \ref{KeyMult} is a restatement of the Grassmann--Berezin (fermionic) calculus \cite{Berezin66, Robinson99} in a coordinate form with use of the hyper-Pfaffian forms.
\end{remark}

\subsection{Hyper-Pfaffian \texorpdfstring{$\SI$}{TEXT}-varieties are well-defined} \label{sec:ExistenceHPfa}
The following theorem provides us with an explicit description of the components $\HHPf^{(m,l)}_{n,p}$ of hyper-Pfaffian $\SI$-varieties. The construction is consistent with definition \ref{HyperDef} (as well as definition \ref{PfDef} for Pfaffian $\SI$-varieties); the theorem proves the existence and uniqueness of the hyper-Pfaffian $\SI$-varieties.

\begin{theorem}\label{thm:HPfComponents}(Set-theoretical description of $\HPf^{(m,l)}(V)$)
We have the following explicit description of the components $\HHPf^{(m,l)}_{n, p} \subseteq \bigwedgem{p} V_{n,p}$:
\begin{itemize}
    \item If $p<m$, then $\HHPf^{(m,l)}_{n, p} = \bigwedgem{p}V_{n, p}$.
    \item If $p=m$, then $\HHPf^{(m,l)}_{n, m}$ is given by the equations $\hpf^{(m;l)}_{I}$ for all $m\cdot l$-subsets $I$ of the set $[n, m]$:
    $$\HHPf^{(m,l)}_{n, m} = 
    \begin{cases}
    \bigwedgem{m}V_{n,m} &\text{ for }n < m(l-1),\\
    \mathop{\bigcap\limits_{A\subseteq [n,m]}}\limits_{ |A| = ml}\; \{ \hpf^{(m,l)}_{A} = 0\} &\text{ for } n \geqslant m(l-1).
    \end{cases}{}
    $$
    In particular, $\HHPf^{(m,l)}_{m(l-1), m} = \{ \hpf^{(m,l)}_{[m(l-1), m]} = 0\}$ is a hyper surface in $\bigwedgem{m}V_{m(l-1), m}$.
    \item If $p>m$, then $\HHPf^{(m,l)}_{n, p} = \HPf^{(m,l)}(V_{n, p})$ is an intersection of pullbacks of the hyper-Pfaffians $\HHPf^{(m;l)}_{n, m}$. Explicitly, 
    $$\HHPf^{(m,l)}_{n, p} = 
    \begin{cases}
    \bigwedgem{p}V_{n,p} &\text{ for } n < m(l-1),\\
    \mathop{\bigcap\limits_{A\subseteq [n,p]}}\limits_{ |A| = ml}\; \{ \hpf^{(m,l)}_{A|J} = 0\} &\text{ for } n \geqslant m(l-1),
    \end{cases}{}
    $$
    where $$\hpf^{(m;l)}_{A|J} = \sum\limits_{I_1 \sqcup \dots \sqcup I_l = A} \sgn(I_1,\dots, I_l) x_{I_1 J}\dots x_{I_l J} \text{ and } J = [n, p] \backslash A.$$
\end{itemize}
\end{theorem}
\begin{remark}
Classically known that $\KK$-points of the Grassmannian $\PPf^{(2)}_{n,p} = \Gr(p, n+p)$ can be described by Pl\"ucker relations of the form 
$$\sum_{j \in J} \sgn(j, I) x_{I\cup j} x_{J\backslash j} \text{ for all }I\in \bigwedgem{p-1}[n+p], J\in \bigwedgem{p+1}[n+p].$$
However, the ideal generated by these relations (inside $\mathrm{Sym}\left(\bigwedgem{p} \ZZ^{2p} \right)$) is not radical, i.e., the set of these relations is not sufficient to generate $\Gr(p, n+p)$ as a scheme over $\mathrm{Spec}(\ZZ)$.

The same situation happens for the hyper-Pfaffian varieties. The set of equations in Theorem \ref{thm:HPfComponents} defines only $\KK$-points of the varieties. A description of ideals for the hyper-Pfaffians is a non-trivial problem which is related to an explicit description of the plethysms $\mathrm{Sym}^{k} \circ \bigwedgem{p}$.
\end{remark}

We recall from Example (3) in Section \ref{sec:MaxMinSIVar} that for any $(n,p)$-maximal $\SI$-variety $\XX$ the ideals defining components $\XX_{N,P}$ with $(N,P) \> (n,p)$ are pullbacks of the ideal for $\XX_{n,p}$. Therefore the statement of the theorem for $(n,p) \> (m(l-1), m)$ is clear if we prove the rest.  

To get other components of the hyper-Pfaffian variety $\HHPf^{(m,l)}$  we need some computations. Instead of a bulky technical proof in a general case, we exemplify internal strings of the proof technique via the elementary example $\HHPf^{(4,2)}$.

\begin{lemma}\label{SetTheoretical42}
Theorem \ref{thm:HPfComponents} holds for $\HHPf^{(4,2)}$. Explicitly, for any $(n,p)\not \> (4,4)$ we have the equality $\HHPf^{(4,2)}_{n, p} = \bigwedgem{p}V_{n, p}$.
\end{lemma}

The following elementary observation is extremely useful for the proof of the lemma.

\begin{observation}
Assume that for a vector $v = \sum v_{I}e_{I} \in \bigwedgem{4}V_{n, 4}$ there exists $j \in [n, 4]$ such that $v_{I}=0$ if $ j \not \in I $. Then $v \in \HPf^{(4,2)}(V_{n,4}) = \bigcap\limits_{A\subseteq [n,4], |A| = 8}\; \{ \hpf^{(4,2)}_{A} = 0\}$.
\end{observation}
\begin{proof}
The observation follows from the fact that for $\hpf^{(4;2)}_{A}(v)$ to be nontrivial, we need at least two nonzero coordinates of $v$ with non-intersecting indices (this is wrong under the assumption for the element~$j$).
\end{proof}

\begin{proof}[Proof of Lemma \ref{SetTheoretical42}] We begin with the case $n = 4, p=3$. The map $j_{4, 3}^{\dagger}: \bigwedgem{4}V_{4, 4} \rightarrow \bigwedgem{3}V_{4, 3}$ in the coordinate form is given by 
    $$e_{I} \mapsto 
    \begin{cases} 
    (-1)^{\sgn(I, 4)} e_{I}\backslash e_{4} &\text{ if }4 \in I;\\
    0 &\text{ otherwise.}
    \end{cases}$$
    For an arbitrary point $(a_{J}) \in \bigwedgem{3}V_{4, 3}$, the point $(A_{I}) \in \bigwedgem{4}V_{4, 4}$, where
    $$ A_{I}=
    \begin{cases}
    a_{J} &\text{ if }I=J\cup 4;\\
    0 &\text{ otherwise},
    \end{cases}$$
    belongs to the preimage of $(a_{J})$ under $j_{4, 3}^{\dagger}$. Applying the observation for $n =4$, $v = (A_{I})$ and $j = 4$, we see that $(A_{I})$ belongs to the hyper-Pfaffian $\HPf^{(4,2)}(V_{4,4})$. 
    
    Therefore we proved that the variety $\HPf^{(4,2)}(V_{4, 4})$ maps surjectively to $\bigwedgem{3}V_{4, 3}$, i.e., $\HHPf^{(4, 2)}_{4, 3} = \bigwedgem{3}V_{4, 3}$. And, more generally, $\HHPf^{(4,2)}_{4, p} = \bigwedgem{p}V_{4, p}$ for $p\leqslant 3$.
    
    Consider the case of a general $n\geqslant 4$ and $p=3$. Again, for any $(a_{I}) \in \bigwedgem{3}V_{n, 3}$ we can consider the point $(A_{I}) \in \bigwedgem{4}V_{n,4}$ where 
    $$ A_{I}=
    \begin{cases}
    a_{J} &\text{ if }I=J\cup n;\\
    0 &\text{ otherwise}.
    \end{cases}$$
    Applying the observation to the case $v = (A_{I})$ and $j = n$, we can see that the point $(A_{I})$ belongs to the variety $\HPf^{(4;2)}(V_{n, 4}) = \HHPf^{(4,2)}_{n, 4} $. So the projection $\HHPf^{(4,2)}_{n, 4} \rightarrow \bigwedgem{3}V_{n, 3}$ is surjective, i.e., $\HHPf^{(4,2)}_{n,p} = \bigwedgem{p}V_{n,p}$ for any $n \geqslant 0$ and $p \leqslant 3$.
    
    We note that the map $i^{\dagger}_{3,4}: \bigwedgem{4}V_{4,4} \rightarrow \bigwedgem{4}V_{3,4}$ in the coordinate form is given by 
    $$e_{I} \mapsto 
    \begin{cases} 
    e_{I} &\text{ if } -4 \not\in I;\\
    0 &\text{ otherwise.}
    \end{cases}$$
    The same technique with the map $i^{\dagger}_{3,4}$ as with $j^{\dagger}_{4,3}$ gives the equality $\HHPf^{(4,2)}_{3,4} = \bigwedgem{4}V_{3,4}$. So $\HHPf^{(4,2)}_{n,4} = \bigwedgem{4}V_{n,4}$ for $n \leqslant 3$. Analogously, for any $(n,p)$ with $n \leqslant 3$ and $p \geqslant 4$ we get the desired equality.
\end{proof}

\subsection{Hyper-Pfaffians and \texorpdfstring{$\GL$}{TEXT}-orbits in exterior powers}

In this section we talk about $\GL$-orbits in the exterior powers. This description is crucial for Theorem \ref{thm:InclToHPf}.

Consider the exterior power $\bigwedgem{p}V$, where $V$ is a $\KK$-vector space of sufficiently high dimension (without loss of generality, we can assume $V = \KK^{\infty}$). One can ask about classification of $\GL(V)$-orbits in this space. Theorem 6 of \cite{DraismaSnowdenEtAl20} implies that such orbits are related to the decomposition type of tensors inside the exterior power $\bigwedgem{p}V$. 

In detail, let $\pi = (\pi_{1} \geqslant \pi_{2} \geqslant \dots)$ be a partition of $p$ and $k$ be a natural number. We call an element $\omega \in \bigwedgem{p}V$ of type $(\pi, k)$ if $\omega$ is equal to a sum of $k$ elements of the space $\bigwedgem{\pi_1}V \wedge \bigwedgem{\pi_2}V \wedge \dots$ and such $k$ is minimal:
$$(\bigwedgem{p}V)_{\pi, k} = \left\{\omega \in \bigwedgem{p}V:\; \omega = \sum\limits_{i = 1\dots k} \omega_{i}, \;\omega_i \in \bigwedgem{\pi_1}V \wedge \bigwedgem{\pi_2}V \wedge \dots, k\text{ is minimal} \right\}.$$

For example, an element $\omega \in \bigwedgem{p}V$ is of type $((p, 0, \dots), 1)$ if and only if $\omega$ is decomposable. So the set of elements of type $((p, 0, \dots), 1)$ coincides with the Grassmannian $\Gr(p, V)$, which is $\GL(V)$-invariant. The general statement is the following. 

\begin{proposition}\label{prop:Orbits} \cite{DraismaSnowdenEtAl20}
Any $\GL(V)$-invariant algebraic subvariety of $\bigwedgem{p}V$ is contained in a Zariski closure of one of the spaces $(\bigwedgem{p}V)_{\pi, k}$, where $\pi = (\pi_{1} \geqslant \pi_{2} \geqslant \dots)$ is a partition of $p$ and the latter space consists of all elements that can be represented as a sum of $k$ elements from $\bigwedgem{\pi_1}V \wedge \bigwedgem{\pi_2}V \wedge \dots$.
\end{proposition}
\begin{remark}
The result of \cite{DraismaSnowdenEtAl20} is far more general: the analogous statement is true for $\GL(V)$-invariant subvarieties of the space $\mathbb{S}_{\mu}V$, where $\mathbb{S}_{\mu}$ is the Schur functor for a partition $\mu$. The analogous result for Proposition \ref{prop:Orbits} in the case $\mathbb{S}_{\mu} = \Sym^{k}$ is proved in \cite{BikDraismaEggermont19}.
\end{remark}

\begin{definition}
We call a partition $\pi = (\pi_{1} \geqslant \dots \geqslant \pi_{j})$ \textit{even} if all of the parts $\pi_{i}$ are even numbers. Otherwise we call it \textit{odd}. 
\end{definition}

\begin{proposition}\label{OddPartition}
For any odd partition $\pi$ and any $k \in \NN$, the set $(\bigwedgem{p}V)_{\pi, k}$ (as well as its Zariski closure) is contained in the varieties $\HPf^{(p,k+1)}(V)$ and $\HPf^{(\dim V - p,k+1), \star}(V^*)$:
$$(\bigwedgem{p}V)_{\pi, k} \subseteq \HPf^{(p,k+1)}(V) \text{ and }(\bigwedgem{p}V)_{\pi, k} \subseteq \HPf^{(\dim V - p,k+1), \star}(V^{*}).$$
\end{proposition}
\begin{proof}
%
We prove the first inclusion only; for the dual variety proof is analogous.

Without loss of generality, we assume that $\pi_{1}$ is odd. Then for any $\alpha \in (\bigwedgem{p}V)_{\pi, k}$ we know that $\alpha \wedge \alpha = 0$ Indeed, $\alpha \wedge \alpha  = \alpha_{1} \wedge \alpha_{1} \wedge \dots $ and $\alpha_{1} \wedge \alpha_{1} = 0$ because $\alpha_{1}$ belongs to $\bigwedgem{\pi_1}V$ with odd $\pi_{1}$. Therefore for any $\beta = \beta^{(1)} + \dots + \beta^{(k)} \in (\bigwedgem{p}V)_{\pi, k}$ by the pigeonhole principle $\beta^{\wedge (k+1)} = 0$. This equality is equivalent to the desired inclusion due to the geometrical description of the hyper-Pfaffian (Corollary \ref{GeoHPf}).

Hyper-Pfaffian and dual hyper-Pfaffian varieties are closed, so we also have the inclusions $\overline{(\bigwedgem{p}V)_{\pi, k}} \subseteq \HPf^{(p,k+1)}(V)$ and $\overline{(\bigwedgem{p}V)_{\pi, k}} \subseteq \HPf^{(\dim V - p,k+1), \star}(V^{*})$ for the Zariski closures.
\end{proof}

\subsection{Any proper Plücker variety is a subset of a two-sided hyper-Pfaffian}

\begin{theorem} \label{thm:InclToHPf}
For any proper $\SI$-variety $\XX$ there exist natural numbers $m,l,r,s$ such that $\XX_{\infty}$ is a closed $\GGL$-stable subset of $\HHPf^{(m,l), (r,s)}_{\infty}$.
\end{theorem}
\begin{proof} Closedness and $\GGL$-invariancy follow from the definition of $\XXX$, so we prove the inclusion $\XXX \subseteq \HHPf^{(m,l), (r,s)}_{\infty}$ only. The idea for the proof is to combine Proposition \ref{prop:Orbits} and \ref{OddPartition}. 

Namely, let $(N, P)$ be a minimal pair such that $\XX_{N, P} \varsubsetneq \bigwedgem{P}V_{N, P}$ and $N, P$ are even. Then by Proposition \ref{prop:Orbits} the variety $\XX_{N, P}$ is contained in Zariski closure of $(\bigwedgem{P}V_{N, P})_{(\Lambda, K)}$ for some $\Lambda$ and $K$. Then $\XX_{N, P+2}$ is contained in the variety $\overline{(\bigwedgem{P+2}V_{N, P+2})_{\Lambda \cup 1 \cup 1, K}}$, where $\Lambda\cup 1\cup 1$ is the partition formed out of $\Lambda$ and two 1's:

\begin{center}
\begin{tikzcd}
  \bigwedgem{P}V_{N,P} \arrow[rr, hook, "\_\wedge e_{P+1}" ] & & \bigwedgem{P+1}V_{N,P+1} \arrow[rr, hook, "\_\wedge e_{P+2}"]  & & \bigwedgem{P+2}V_{N,P+2} \\
  \overline{(\bigwedgem{P}V_{N,P})_{\Lambda, K}} \arrow[draw=none]{u}[sloped,auto=false]{\subseteq} \arrow[rrrr, hook, "\_\wedge e_{P+1}\wedge e_{P+2}"]  & &   & & \overline{(\bigwedgem{P+2}V_{N,P+2})_{\Lambda\cup 1 \cup 1, K}} \arrow[draw=none]{u}[sloped,auto=false]{\subseteq}\\
  \XX_{N, P} \arrow[draw=none]{u}[sloped,auto=false]{\subseteq} \arrow[rrrr] & &  & & \XX_{N, P+2} \arrow[draw=none]{u}[sloped,auto=false]{\subseteq} 
\end{tikzcd}
\end{center}

The partition $\Lambda\cup 1\cup 1$ is odd (regardless of the $\Lambda$'s parity), so by Proposition \ref{OddPartition} we get the chain of inclusions $$\XX_{N, P+2} \subseteq \overline{(\bigwedgem{P+2}V_{N,P+2})_{\Lambda\cup 1 \cup 1, K}} \subseteq \HHPf^{(P, K+1)}_{N, P+2}.$$
This inclusion and the (maximality in the) definition of hyper-Pfaffian $\SI$-varieties imply the inclusion of the limiting varieties: 
$$\XX_{\infty} \subseteq \HHPf_{\infty}^{(P, K+1)}.$$

For the dual hyper-Pfaffians, we have $\XXX \subseteq \HHPf_{\infty}^{(N, K+1), \star}$. Finally, $\XXX \subseteq \HHPf_{\infty}^{(P, K+1),(N, K+1)}$.
\end{proof}


\subsection{Hyper-Pfaffian filtration}\label{GeneralFiltration}
First, we prove an elementary lemma about inclusions for hyper-Pfaffian varieties.

\begin{lemma}\label{lemma:filtration}
Consider a vector space $V$. Then we have an inclusion 
$$\HPf^{(m,l)}(V) \subseteq \HPf^{(m, l+1)}(V).$$
\end{lemma}
\begin{proof}
For any set $A$ of cardinality $m(l+1)$ we have the equality
$$\hpf^{(m,l+1)}_{A}({\bf x}) = \sum\limits_{I \ni 1} \pm\, x_{I} \cdot \hpf^{(m, l)}_{A\backslash I}({\bf x}),$$
where the summation runs over all sets $I \in \bigwedgem{m}(A)$ containing the element $1$. This equality proves the desired inclusion of varieties.
\end{proof}

\begin{proposition}\label{prop:filtration}
For any finite dimensional vector space $V$ we have the exhaustive separable filtration
$$ \{0\} = \HPf^{(m, 1)}(V) \subseteq \dots \subseteq \HPf^{(m, l)}(V) \subseteq \HPf^{(m,l+1)}(V) \subseteq \dots \subseteq \bigwedgem{m}V.$$
\end{proposition}
\begin{proof}
The existence of the filtration follows from Lemma \ref{lemma:filtration}.

As mentioned in Remark \ref{Root}(3), the hyper-Pfaffian form $\hpf^{(m, l)}_{[ml]}$ is an $m$-th root from the determinant $\det_{[ml]}$ on the set of completely antisymmetric tensors. So for a sufficiently big $L$ the variety $\HPf^{(m, L)}(V)$ coincides with the ambient space $\bigwedgem{m}V$. The number $L$ depends on dimension of the space $V$.
\end{proof}

In light of the proposition, Theorem \ref{thm:InclToHPf} can be reformulated as follows.
\begin{corollary}\label{cor:finiterankPlu}
Every proper $\SI$-variety has a finite rank, i.e., for any proper $\XX$ there exists a natural number $N$ such that all $N\times N$ determinants are identically zero on $\XXX$.
\end{corollary}

\begin{remark}
The filtration of Proposition \ref{prop:filtration} is exhaustive for a finite dimensional vector spaces $V$ only. Indeed, from Proposition \ref{OddPartition} we see that for an odd partition $\pi$ the space $\overline{(\bigwedgem{m}V)_{\pi, k}}$ is contained in the hyper-Pfaffian $\HPf^{(m,k+1)}(V)$ as a scheme. However, we do not have such embeddings for even partitions. For instance, even for the simplest case of $\pi = (2,2)$, none of the hyper-Pfaffian forms $\hpf^{(4, l)}$ belongs to the ideal corresponding to the affine scheme $\overline{(\bigwedgem{4}V)_{(2,2), 1}}$. 
\end{remark}

\section{Noetherianity proof}

\begin{theorem} \label{thm:HPfNoeth}
For any natural numbers $m,k,r,s$, the variety $\HHPf^{(m,l),(r,s)}_{\infty}$ is $\GGL$-noetherian.
\end{theorem}
It's curious that the proof of Draisma--Eggermont with minor changes is applicable in the general situation. Because of this, here we present a compact version of the proof (keeping the notation of \cite{DraEgger}).

Following \cite{DraEgger}, we use the general lemma.

\begin{lemma} \label{OneElem}
Let $\omega \in \WedgeDU$ and suppose there exist elements $g_1, g_2 \in \GGL$ such that $F_1(g_1\omega)\neq 0$ and $F_2(g_2\omega)\neq 0$ for some polynomial functions $F_{1,2}$ on $\WedgeDU$. Then there exists an element $g \in \GGL$ such that $F_{1}(g\omega)\neq0$ and $F_{2}(g\omega)\neq 0$. Moreover, the element $g$ can be found in the form $g = \lambda g_1 + \mu g_2$ for some $\lambda, \mu \in K$.
\end{lemma}
\begin{proof} Consider the function $g := g (\lambda, \mu) = \lambda g_1 + \mu g_2$ and the set
$$\mathcal{F}:= \{(\lambda, \mu) \in K^2:\; g \not\in \GGL,\; F_{1}(g\omega) = 0 \text{, and }F_{2}(g\omega) = 0\}.$$
The set $\mathcal{F}$ is Zariski-closed by definition. The polynomial $F_{1}(g\omega)$ in variables $\lambda$ and $\mu$ has a coefficient $F_{1}(g_1\omega)$ for a highest degree monomial containing the variable $\lambda$ only. This coefficient in nonzero by the assumption, therefore the polynomial $F_{1}(g \omega)$ is also nonzero. The same logic for the polynomial $F_2(g\omega)$ and the variable $\mu$ proves nonzeroness of this polynomial. Therefore $\mathcal{F}$ is a proper Zariski-closed subset of $K^2$. 

Any point $(\lambda, \mu)\not\in \mathcal{F}$ gives the desired element $g \in \GGL$. 
\end{proof}

\begin{proof}[Proof of Theorem \ref{thm:HPfNoeth}]

We fix numbers $m$ and $r$ and proceed by induction on $l$ and $s$. 

The base of the induction, the case $l = 1$ or $s = 1$, is clear. Indeed, the $\GGL$-orbit of the polynomial $\hpf^{(m;1)}_{[m]}({\bf x}) = x_{12\dots m}$ contains all coordinate variables $x_{I}$, $I \in \bigwedgem{m} \NN$. The intersection of these polynomials is the point $0$.

From now on we assume that $l, s > 1$. 

\textit{\textbf{Step 1}: decomposition and the variety $Z$.} By Lemma \ref{lemma:filtration} we have the decomposition
$$\HPf^{(m,l+1),(r,s+1)}_{\infty} = \HPf^{(m,l+1),(r,s)}_{\infty} \cup \HPf^{(m,l),(r,s+1)}_{\infty} \cup Z'_{m,r},$$
where $Z' := Z'_{m,r}$ is the set of all elements $\omega$ such that there exist $g_1, g_2 \in\GGL$ with $\hpf^{(m,l)}(g_1 \omega) \neq 0$ and $\hpf^{(r,s), \star}(g_2 \omega) \neq 0$:
$$Z':= \{\omega \in \HHPf^{(m,l+1),(r, s+1)}_{\infty}:\;\hpf^{(m,l)}(g_1 \omega) \neq 0\text{ and } \hpf^{(r,s),\star}(g_2 \omega) \neq 0\text{ for some }g_{1}, g_{2}\in \GGL\}.$$
Applying Lemma \ref{OneElem} for the case $F_{1} = \hpf^{(m,l)}$ and $F_{2} = \hpf^{(r,s),\star}$, we get that $Z' = \GGL \cdot Z$ where
$$Z = \{\omega \in \HHPf^{(m,l+1),(r,s+1)}_{\infty}: \hpf^{(m,l)}(\omega)\neq 0 \text{ and }\hpf^{(r,s),\star}(\omega) \neq 0\}.$$
For the induction step it is enough to show the $\GGL$-noetherianity of the set $Z$. 

\textit{\textbf{Step 2}: the subgroup $H$.} We will prove that the set $Z$ is noetherian for a certain subgroup $H$ of $\GGL$. Let us define this subgroup $H$.

For any $S \subseteq \ZZX$, let $\GL_{S}$ be the subgroup of $\GGL$ that fixes $e_{t}\in V_{\infty}$ with $t\not\in S$. Then we define the subgroup $H$ as the following product
$$H:= \GL_{(-\infty, -ml+m)}\times \GL_{(rs-r, +\infty)},$$
where $(a,b)$ stands for the set $\{x \in \NN: a < x < b\}$. 

The group $H$ stabilizes $Z$. Indeed, the coordinates $x_{I}$ appearing in the forms $\hpf^{(m,l)}$ and $\hpf^{(r,s), \star}$ satisfy the condition $I \subseteq [-ml+m, rs-r]$.

\textit{\textbf{Step 3}: A ``good'' subspace.} To prove that $Z$ is $H$-Noetherian, we embed it into a bigger $H$-Noetherian subspace of $\WedgeDU$.

\begin{definition}
A subset $I\subseteq \ZZX$ is \textit{good} (with respect to $(m, l, r, s)$) if $|I \cap \ZZ_{<0}|$ and $|I^{c}\cap \ZZ_{>0}|$ are finite of the same cardinality, and both $I\cap\mathbb{Z}_{\leqslant -ml+m-1}$ and $I^c \cap \mathbb{Z}_{\geqslant rs-r}$ have cardinality at most 1.

We denote by $(\Wedge)_{\text{good}}$ the subspace of $\Wedge$ spanned by good coordinates. The dual space $\Wedgegood$ is naturally $H$-equivariantly embedded into $\WedgeDU$. 
\end{definition}

\begin{lemma}
The topological space $\Wedgegood$ with the Zariski topology is $H$-Noetherian.
\end{lemma}
The proof of this Lemma is literally the same as for \cite[Lemma 6.5]{DraEgger}: we can embed the space $\Wedgegood$ into the space $A_{a,b,c,d}$, where 
$$A_{a,b,c,d} = \left(\mathrm{Mat}_{\NN\times\NN}\right)^a \times \mathrm{Mat}_{\NN\times b}\times \mathrm{Mat}_{c \times \NN}\times \KK^d.$$
Then the theorem of Draisma--Eggermont \cite[Theorem 1.5]{DraEgger} states that the space $A_{a,b,c,d}$ is $\GGL\times\GGL$-noetherian, so the space $\Wedgegood$ is. 
\begin{remark}
General statement with the subspace of $\Wedge$ spanned by all ``$S$-good'' coordinates for some finite subset $S \subset \mathbb{N}$ (in the considered case $S = [-ml+m, rs-r]$ ) is true as well. The proof works in the general case.
\end{remark}

\textit{\textbf{Step 4}: the injective projection $Z \rightarrow \Wedgegood$.} Finally, we prove that the natural projection from $Z$ to the good subspace $\Wedgegood$ (the projection forgets non-good coordinates) is injective.

\begin{claim}
On $Z$, each coordinate $x_{I}$ can be expressed as a rational function of the good coordinates, whose denominator has factors $\hpf^{(m,l)}$ and $\hpf^{(r,s), \star}$ only.
\end{claim}

It is classically known that coordinates on the space $\Wedge$ are in one-to-one correspondence with Young diagrams, the classical reference is \cite[Chapter 9]{MiwaJinboJimboDate00}, the more contemporary exposition is in \cite{Rios-Zertuche13}. The idea for the proof is the induction on coordinates $\{x_{I},\; I \in \bigwedgem{\frac{\infty}{2}}\NN\}$ with respect to the partial order $<$ coming from the partial order $\prec$ on the set of Young diagrams. For instance, the relation $\emptyset \prec (2) \prec (2,1)$ for Young diagrams translates to the relation $x_{1234\dots} < x_{-1134\dots} < x_{-2134\dots}$ for the coordinates. 

In detail, consider a non-good coordinate $x_{I}$ such that all smaller coordinates are good or satisfy the claim.

If $|I\cap \mathbb{Z}_{\leqslant -ml+m-1}|=1$, then denote by $\mathcal{I}$ any subset of $\ZZX$ such that the set $I$ is initial subinterval of $\mathcal{I}$ (with respect to the natural order on $\ZZX$).  Then the hyper-Pfaffian form $\hpf^{(m,l+1)}_{\mathcal{I}}$ looks as follows
$$ \hpf^{(m,l+1)}_{\mathcal{I}} = x_{I} \cdot \hpf^{(m,l)}_{\mathcal{I}\backslash I}+Q$$
with all coordinates in $Q$ are strictly smaller than $x_I$. The polynomial $\hpf^{(m, l+1)}_{\mathcal{I}}$ is zero on $Z$ (as a shift of the polynomial $\hpf^{(m,l+1)}({\bf x})$ which is identically zero on $Z \subset \HHPf^{(m, l+1),(r, s+1)}_{\infty}$), therefore the following equality holds true on $Z$:
$$x_{I} = -\frac{Q}{\hpf^{(m,l)}_{\mathcal{I}\backslash I}}.$$

The case $|I\cap \mathbb{Z}_{\geqslant rs-r}|=1$ is treated similarly with use of dual hyper-Pfaffian forms.
\end{proof}

\begin{corollary}\label{cor:MainThm}
Theorem \ref{Thm:MainThm} holds true.
\end{corollary}
\begin{proof}
Theorem \ref{thm:InclToHPf} says that for any proper $\SI$-variety $\XX$ the limiting variety $\XXX$ is a subset of a $\GGL$-noetherian variety $\HHPf^{(m,l),(r,s)}_{\infty}$ (Theorem \ref{thm:HPfNoeth}). Therefore the $\GGL$-stable variety $\XXX$ is given by $\GGL$-orbits of finitely many polynomial equations inside $\WedgeDU$.

We can find $n_0$ and $p_0$ such that the $\GGL$-orbits of the equations of $\XX_{n_0, p_0}$ define $\XXX$. Therefore using $i$- and $j$-maps we can see that all $\XX_{n,p}$ are defined by $\GGL$-translations of the polynomials for $\XX_{n_0, p_0}$.
\end{proof}

\section{Examples}\label{sec:ExplHPfa}

\subsection{\texorpdfstring{$\Pf^{(2)}$}{TEXT}}
We recall from Proposition \ref{prop:components} that we always have the following inclusions for the $(n,p)$-component of $\SI$-variety:
$$\YY_{n,p}^{\min}:= \overline{\bigcup_{j} j(\YY_{n, p-1})} \subseteq \YY_{n, p} \subseteq \bigcap_{j^{\dagger}} (j^{\dagger})^{-1}(\YY_{n,p-1}) =: \YY_{n,p}^{\max},$$
where the right intersection runs over the orbit of the map $j_{n,p-1}$ under the $\GL_{n, p-1}\times\GL_{n, p}$-action and the left union runs over the orbit of $j_{n,p-1}^{\dagger}$ under the same group.

Let us assume that $\YY_{4,2} = \Pf^{(2)}(V_{4,2})$. The next lemma describes the minimal bound for the $(4,3)$-component.

\begin{lemma}\label{lemma:4,3min}
\begin{enumerate}
    \item[(a)] The variety $\YY^{\min}_{4,3} := \overline{\bigcup_{j} j(\YY_{4, 2})}$ coincides with $\overline{\bigwedgem{2}_{(2)}\wedge\bigwedgem{1}} (V_{4,3})$, where $$\bigwedgem{2}_{(2)}\wedge\bigwedgem{1} (V) := \{\omega \in \bigwedgem{3}(V):\;\omega = (\nu_{1}+ \nu_2)\wedge v \text{ with } \nu_{1,2} \in (\bigwedgem{2}V)_{(1,1),1}\text{ and }v\in \bigwedgem{1}(V) = V\}.$$
    \item[(b)] Generally, if $\XX_{n,p} = \overline{(\bigwedgem{p}V_{n,p})_{\pi, k}}$, then $\XX_{n, P}^{\min} = \overline{\bigwedgem{P}_{\pi, k} \wedge \bigwedgem{1} \wedge \dots \wedge \bigwedgem{1}}(V_{n, P})$ for any $P\geqslant p$.
\end{enumerate}
\end{lemma}
\begin{proof} We prove the first part only; the proof for the second one uses the same idea.

From the formula $j_{4,2}(\nu) = \nu \wedge e_{3}$ the inclusion $\YY^{\min}_{4,3} \subseteq \overline{\bigwedgem{2}_{(2)}\wedge\bigwedgem{1}} (V_{4,3})$ follows.

By one of the definitions of the $\pf^{(2)}({\bf x})$ form, $\nu \in \bigwedgem{2}(V_{4,2})$ belongs to $\Pf^{(2)}(V_{4,2})$ if and only if $\rk(\nu) \leqslant 2$, i.e., $\nu = \nu_1 + \nu_2$ with $\nu_{1,2} \in (\bigwedgem{2}V)_{(1,1),1}$. Also, for any trivector $\nu \wedge v$ with $\nu \in \bigwedgem{2}V$ and $v \in \bigwedgem{1}V$ we can assume that $\nu \in \bigwedgem{2}\langle v \rangle^{\bot}$. Therefore any element $\omega \in \bigwedgem{2}_{(2)}\wedge\bigwedgem{1} (V_{4,3})$ has the form $\nu \wedge v$ with $\rk(\nu) \leqslant 2$ and $\nu \in \bigwedgem{2}\langle v\rangle^{\bot}$, and we have the inclusion $\overline{\bigwedgem{2}_{(2)}\wedge\bigwedgem{1}} (V_{4,3}) \subseteq \YY^{\min}_{4,3}$.
\end{proof}

\begin{proposition}\label{prop:HoriNonEqua}
If $\YY_{4,2} = \Pf^{(2)}(V_{4,2})$, then $\YY_{4,3}^{\min} \varsubsetneq \YY_{4,3}^{\max}$. Moreover, $\YY_{4,3}^{\min} \varsubsetneq (\bigwedgem{3}V_{4,3})_{(2,1), 2}$ and $\YY_{4,3}^{\max} \neq (\bigwedgem{3}V_{4,3})_{(2,1), 2}$.
\end{proposition}
\begin{proof} For this proof we denote $X:= (\bigwedgem{3}V_{4,3})_{(2,1), 2}$. We use the description of $\YY_{4,3}^{\min}$ from Lemma \ref{lemma:4,3min}.

The trivector $(e_{-4}\wedge e_{-3} + e_{-2}\wedge e_{-1} + e_{1}\wedge e_{2})\wedge e_{3}$ belongs to $X\backslash \YY_{4,3}^{\max}$ and $X \backslash \YY_{4,3}^{\min}$. The inclusion $\YY_{4,3}^{\min} \subset X$ follows from the definitions. Finally, the trivector $e_{-4}\wedge e_{-3} \wedge e_{-2} + e_{-1}\wedge e_{1} \wedge e_{2}$ belongs to $\YY_{4,3}^{\max}$, but not to $\YY_{4,3}^{\min}$.
\end{proof}

\subsection{\texorpdfstring{$\HPf^{(4,2)}$}{TEXT}}

From Proposition \ref{prop:components} that we always have the following inclusions for the $(n,p)$-component:
$$\ZZZ_{n,p}^{\min}:= \overline{\bigcup_{i} i(\ZZZ_{n-1, p})} \subseteq \ZZZ_{n, p} \subseteq \bigcap_{i^{\dagger}} (i^{\dagger})^{-1}(\ZZZ_{n-1,p}) =: \ZZZ_{n,p}^{\max},$$
where the right intersection runs over the orbit of the map $i_{n-1,p}$ under the $\GL_{n-1, p}\times\GL_{n, p}$-action and the left union runs over the orbit of $i_{n-1,p}^{\dagger}$ under the same group.

Let us assume that $\ZZZ_{4,4} = \HPf^{(4,4)}(V_{4,4})= \{\hpf^{(4,2)}_{[4,4]} = 0\}$, where $$\hpf^{(4,2)}_{[8]}({\bf x}) = \sum_{I\sqcup J = [8], |I| = |J| = 4} \sgn(I,J) x_{I}x_{J} = x_{1234}x_{5678} - x_{1235}x_{4678} + \dots + x_{1678}x_{2345}.$$
This is the smallest ``non-Pfaffian'' example of a hyper-Pfaffian variety. Due to Proposition \ref{KeyMult}, the variety $\HPf^{(4,2)}(V)$ is the set of all elements $\omega$ in $\bigwedgem{4}V$ satisfying $\omega \wedge \omega = 0$. Therefore the variety $\ZZZ_{5,4}^{\max}$ coincides with $\HPf^{(4,2)}(V_{5,4}) = \{\omega \in \bigwedgem{4}V_{5,4}:\; \omega\wedge \omega = 0\}$.

\begin{proposition}\label{prop:VertNonEqua}
If $\ZZZ_{4,4} = \HPf^{(4,4)}(V_{4,4})$, then $\ZZZ_{5,4}^{\min} \varsubsetneq \ZZZ_{5,4}^{\max}$.
\end{proposition}
\begin{proof}
We can see that dimensions of the varieties are different, or instead just argue that $$e_{-5}\wedge e_{-4} \wedge e_{-3} \wedge e_{-2} + e_{-1}\wedge e_{1} \wedge e_{2} \wedge e_{3} + e_{-5}\wedge e_{-4} \wedge e_{-3} \wedge e_{-1} + e_{-2}\wedge e_{1} \wedge e_{2} \wedge e_{3} + e_{-5}\wedge e_{-2} \wedge e_{-1} \wedge e_{4} \in \ZZZ_{5,4}^{\max}\backslash \ZZZ_{5,4}^{\min}$$
via direct computations.
\end{proof}

\subsection{\texorpdfstring{$\SI$}{TEXT}-varieties which are not Pl\"ucker, and unbounded Pl\"ucker varieties} \label{sec:SINotPlu}
The following proposition explains which of the Pfaffian $\SI$-varieties are Pl\"ucker.
\begin{proposition}\label{prop:PfPlu}
\begin{enumerate}
    \item[$(1)$] The $\SI$-varieties $\PPf^{(1)}$ and $\PPf^{(2)}$ are Pl\"ucker, i.e., these varieties are preserved by the $\star$-symmetry.
    
    \item[$(2)$] The $\SI$-variety $\PPf^{(r)}$ for $r \geqslant 3$ are not Pl\"ucker. In other words, the $\star$-symmetry does not preserve $\PPf^{(r)}$ for $r \geqslant 3$.
\end{enumerate}
\end{proposition}
\begin{proof} The first statement follows from a direct computation and Lemma \ref{lemma:4,3min} (b).

We prove the second statement for $r = 3$ only; the general case is analogous. From Lemma \ref{lemma:4,3min} (b) we see that $\eta \in \bigwedgem{4}V_{4,3}$ lies in the $(4,3)$-component of $\PPf^{(3)}$ if and only if $\eta = (w_{1}\wedge w_{2} +w_{3}\wedge w_{4} +w_{5}\wedge w_{6})\wedge w_{7}$ for some $w_{i} \in V_{4,3}$. A direct calculation shows that $$\star \big((e_{-4}\wedge e_{-3} + e_{-2}\wedge e_{-1} + e_{1}\wedge e_{2})\wedge e_{3}\big) = e_{-2}\wedge e_{-1} \wedge e_{1}\wedge e_{2} + e_{-4}\wedge e_{-3} \wedge e_{1}\wedge e_{2} + e_{-4}\wedge e_{-3} \wedge e_{-2}\wedge e_{-1}$$
But from the same Lemma we see that $\omega \in \bigwedgem{4}V_{3,4}$ belongs to the $(3,4)$-component of $\PPf^{(3)}$ if and only if $\omega = (v_{1}\wedge v_{2} + v_{3}\wedge v_{4})\wedge v_{5} \wedge v_{6}$ for some $v_{i} \in V_{3,4}$. A contradiction.
\end{proof}

Generalizing the proposition we can see the following result. We recall that for $l=1$ and any $m$ the hyper-Pfaffian $\HHPf^{(m,1)} = \{0\}$ is trivial.

\begin{theorem}
A nontrivial hyper-Pfaffian $\SI$-variety $\HHPf^{(m,l)}$ is Pl\"ucker if and only if $l=2$. Moreover, the only nontrivial bounded hyper-Pfaffian Pl\"ucker variety is $\HHPf^{(2,2)} = \PPf^{(2)} = \GGr$.
\end{theorem}
\begin{proof}
We recall that the $\SI$-variety $\HHPf^{(m,l)}$ is defined as the only $(m(l-1), m)$-maximal variety with $(m(l-1), m)$-component equal to the corresponding hyper-Pfaffian (Theorem \ref{thm:HPfComponents}). Then the statement follows: if the hyper-Pfaffian $\SI$-variety is nontrivial ($l \neq 1$), then only for $l=2$ we have the equality $m(l-1) = m$ which is equivalent to the $\star$-symmetry for $(n,p)$-maximal $\SI$-varieties. 

From Theorem \ref{thm:HPfComponents} we see that for $m \geqslant 4$ the varieties $\HHPf^{(m,2)}_{n, 2}$ coincide with $\bigwedgem{2}V_{n,2}$. Therefore the corresponding varieties are unbounded.
\end{proof}

\end{document}